\documentclass[12pt]{article}

\usepackage{amsthm}
\usepackage{amsfonts}
\usepackage{amsmath}
\usepackage[round]{natbib}
\usepackage{graphicx}
\usepackage{vmargin, enumerate}
\usepackage{setspace}

\newcommand{\R}{\mathbb{R}}
\newcommand{\Z}{\mathbb{Z}}
\newcommand{\N}{{\mathbb N}}

\newcommand{\E}[1]{{\mathbf E}\left\{#1\right\}}
\newcommand{\e}{{\mathbf E}}

\newcommand{\p}[1]{{\mathbf P}\left\{#1\right\}}

\newcommand{\psup}[2]{{\mathbf P}^{#1}\left\{#2\right\}}

\newcommand{\I}[1]{{\mathbf 1}_{[#1]}}
\newcommand{\set}[1]{\left\{ #1 \right\}}
\newcommand{\Cprob}[2]{\mathbf{P}\set{\left. #1 \; \right| \; #2}}
\newcommand{\probC}[2]{\mathbf{P}\set{#1 \; \left|  \; #2 \right. }}

\newcommand{\bea}{\begin{eqnarray}}
\newcommand{\eea}{\end{eqnarray}}
\newcommand{\eal}{\nonumber\\}

\newtheorem{thm}{Theorem}
\newtheorem{lem}[thm]{Lemma}

\newtheorem{cor}[thm]{Corollary}

\newcommand{\ssd}{\mathcal{D}}

\newcommand{\ssi}{\mathcal{I}}
\newcommand{\sk}{\mathcal{K}}

\newcommand{\ssf}{\mathcal{S}}
\newcommand{\st}{\mathcal{T}}
\newcommand{\xf}{\mathcal{X}}

\newcommand{\zf}{\mathcal{Z}}

\newcommand{\Bin}{\mbox{\upshape Bin}}

\newcommand{\eventn}[1]{\mbox{\textsl{#1}}}

{ \end{list} }

\begin{document}

\title{Ballot theorems for random walks with finite variance}
\author{L.~Addario-Berry, B.A.~Reed}
\maketitle

\begin{abstract}
We prove an analogue of the classical ballot theorem that holds for any mean zero random walk with positive but finite variance. Our result is best possible: we exhibit examples demonstrating 
that if any of our hypotheses are removed, our conclusions may no longer hold. 
\end{abstract} 

\section{Introduction}\label{sec:introduction}
The classical ballot theorem, proved by \citet{bertrand1887solution}, states that in an election 
where one candidate receives $p$ votes and the other receives $q < p$ votes, the probability 
that the winning candidate is in the lead throughout the counting of the ballots is precisely 
\[
\frac{p-q}{p+q},
\]
assuming no one order for counting the ballots is more likely than another. Viewed as a statement about 
random walks, Bertrand's ballot theorem states that given a symmetric simple random walk $S$ and integers 
$n,k$ with $0 < k \leq n$ and with $n$ and $k$ of the same parity, 
        \[\p{S_i > 0~\forall~0<i<n\vert S_n=k} = \frac{k}{n}.\]
The standard approach to extending Bertrand's ballot theorem is most 
easily explained by first transforming the statement, letting $S_i'=i-S_i$ for $i=1,2,\ldots,n$. 
$S'$ is an {\em increasing} random walk, and the classical ballot theorem states that 
\[\p{S_i' < i~\forall~0<i<n~\vert~S_n'=n-k} = \frac{k}{n}.\]
One may then ask: for what other increasing stochastic processes does the same result hold? This question has 
been well-studied; much of the seminal work on the subject was done by \cite{takacs62ballot, 
takacs62generalization, takacs62time, takacs63distribution, takacs64combinatorial, takacs64fluctuations, takacs67comb}. 
The most general result to date is due to \cite{kallenberg99ballot} 
(see also \cite[][Chapter 11]{kallenberg03foundations}). 

If, rather than transforming $S$ into an increasing stochastic process, on takes the fact that $S_n/\sqrt{n}$ 
converges in distribution to a normal random variable as a starting point, a different generalization 
of Bertrand's ballot theorem emerges. It turns out that if $\e{X}=0$ and $0 < \E{X^2} < \infty$, then 
for a simple random walk $S$ with step size $X$, 
 \begin{equation}\label{eq:intuit_ballot}
 	\p{S_i > 0~\forall~1 \leq i \leq n \vert S_n=k}=\Theta\left(\frac k n\right)\mbox{ for all }k \mbox{ with } 0<k=O(\sqrt{n}).
\end{equation}
(This is a slight misrepresentation; when we consider random variables that do not live on a lattice, 
the right conditioning will in fact be on an event such as $\{k \leq S_n < k+1\}$ or something similar. For the moment, 
we ignore this technicality, presuming for the remainder of the introduction 
that $X$ is integer-valued and $\p{X=1}>0$, say.)


Furthermore, and (to the authors) more surprising, it turns out that this result is essentially best possible. 
We provide examples which demonstrate that if either $\E{X^2}=\infty$ or $k \neq O(\sqrt{n})$, no equation such as (\ref{eq:intuit_ballot}) can be expected to hold. The 
philosophy behind these examples can be explained yet another perspective on ballot-style results.
One may ask: what are sufficient conditions on the structure of a multiset $\ssf$ of $n$ numbers summing to some 
value $k$ to ensure that, in a uniformly random permutation of the set, 
all partial sums are positive with probability of order $k/n$? 
\label{page:multiset}
This latter perspective is philosophically 
closely tied to work of \cite{andersen53fluctuations,andersen54fluctuations}, \cite{spitzer56combinatorial} and others on the amount 
of time spent above zero by a conditioned random walk and on related questions. 
Our procedure for constructing the examples showing that our result is best possible is:
\begin{enumerate}
	\item construct a multiset $\mathcal{M}$ (whose elements sum to $k$, say) and for for which the bounds predicted by the ballot theorem fail;
	\item find a random variable $X$ and associated random walk $S$ for which, given that $S_n=k$, the random multiset $\{X_1,\ldots,X_n\}$ is 
		likely close to $\mathcal{M}$ in composition.
\end{enumerate}
The details involved in carrying out this procedure end up being somewhat involved. 

\subsection*{Outline}
In Section \ref{sec:general} we prove our positive result, the formalization of (\ref{eq:intuit_ballot}). 
Section \ref{sec:counterexamples} contains the examples which show that 
(\ref{eq:intuit_ballot}) may fail to hold if either $\E{X^2}=\infty$ or if $k \neq O(\sqrt{n})$. 
Finally, in Section \ref{sec:conclusion} we briefly discuss 
directions in which the current work might be extended and related potential avenues of research. 

\section{A general ballot theorem}\label{sec:general}
Throughout this section, $X$ is a mean zero random variable with $0 < \E{X^2} < \infty$ and $S$ is a simple random walk with step size $X$. 
Before stating our ballot theorem, we introduce a small amount of terminology. 
We say $X$ is a {\em lattice random variable} with period $d>0$ 
if there is a constant $z$ such that $dX-z$ is an integer random variable and $d$ is the smallest positive real number 
for which this holds; in this case, we say that the set $\mathbb{L}_X=\{(n+z)/d:n \in \Z\}$ is {\em the lattice of} $X$. 
A real number $A$ is {\em acceptable} for $X$ if $A \geq 1/d$ (if $X$ is a lattice random variable), or $A > 0$ (otherwise). 
	\begin{thm}\label{thm:gbt_new}
	For any $A$ which is acceptable for $X$, there is $C > 0$ depending only on $X$ and $A$ such that 
	for all $n$, for all $k>0$,
	\[\p{k \leq S_n < k+A,S_i >0 ~\forall~ 0 < i < n} \leq \frac{C\max\{k,1\}}{n^{1/2}},\]
	and for all $k$ with  with $0 < k \leq \sqrt{n}$, 
	\[\p{k \leq S_n < k+A,S_i >0 ~\forall~ 0 < i < n} \geq \frac{\max\{k,1\}}{Cn^{1/2}},\]
\end{thm}
Before proving Theorem \ref{thm:gbt_new}, we collect a handful of results which we will use in the course of the proof. 
First, we will use a rather straightforward result on how ``spread out'' 
sums of independent identically distributed random variables become. The version we present is a simplification of Theorem 1 in  \cite{kesten72spread}:
\begin{thm}\label{thm:spread}
        For any family of independent identically distributed real random variables $X_1,X_2,\ldots$ 
        with positive, possibly infinite variance and associated partial sums $S_1,S_2,\ldots,$ 
        there is $c>0$ depending only on the distribution of $X_1$ such that for all $n$, 
                \[\sup_{x \in \R} \p{x \leq S_n \leq x+1 } \leq c/\sqrt{n}.\]
\end{thm}
We will also use the following lemma, Lemma 3.3 from \citep{pemantle95critical}:
\begin{lem}\label{lem:pem_per}
	For $h \geq 0$, let $T_h$ be the first time $t$ that $S_t < -h$. Then there are constants $c_1,c_2,c_3$ such that for all $n$:
	\begin{itemize}
		\item[(i)] for all $h$ with $0 \leq h \leq \sqrt{n}$, $\p{T_h \geq n} \geq c_1\cdot \max\{h,1\}/\sqrt{n}$; 
		\item[(ii)] for all $h$ with $0 \leq h \leq \sqrt{n}$, $\E{S_n^2 \vert T_h > n} \leq c_2 n$; and 
		\item[(iii)] for all $h\geq 0$, $\p{T_h \geq n} \leq c_3 \cdot\max\{h,1\}/\sqrt{n}$.
	\end{itemize}
\end{lem}
(In \citep{pemantle95critical}, (ii) was only proved with $h=0$, but an essentially identical proof yields the above formulation.) 
We will use the following easy corollary of Lemma \ref{lem:pem_per} in proving the lower bound of Theorem \ref{thm:gbt_new}. 
\begin{cor}\label{cor:pem_per}
There exists $\epsilon > 0$ such that for all $n$ and all $h$ with $0 \leq h \leq \sqrt{n}$, 
\[\E{S_n^2~\vert~T_h > n, S_n \geq \epsilon \sqrt{n}} \leq 3c_2 n.\]
\end{cor}
\begin{proof}
Choose $\epsilon > 0$ such that for all $n$, $\p{S_n \geq \epsilon \sqrt{n}} > 1/3$. By the FKG inequality, $\probC{S_n \geq \epsilon \sqrt{n}}{T_0 > n} \geq 1/3$. 
We thus have 
\begin{eqnarray}
	\E{S_n^2~\vert~T_h > n, S_n \geq \epsilon \sqrt{n}} 	& = & \frac{\E{S_n^2 \I{S_n \geq \epsilon \sqrt{n}}~\vert~T_0 > n}}{\probC{S_n \geq \epsilon \sqrt{n}}{T_0 > n}} \nonumber \\
											& \leq & 3 \E{S_n^2 \I{S_n \geq \epsilon \sqrt{n}}~\vert~T_0 > n} \leq 3\E{S_n^2~\vert~ T_0 > n}, \nonumber
\end{eqnarray}
and Lemma \ref{lem:pem_per} (ii) completes the proof. 
\end{proof}
Finally, in proving the lower bound of Theorem \ref{thm:gbt_new}, we will use a local central limit theorem. The following 
is a weakening of Theorem 1 from \cite{stone65local}.
\begin{thm}[\cite{stone65local}]
\label{thm:stone}
Fix any $c > 0$. If $\e{X}=0$ and $0 < \E{X^2} < \infty$ then for any $h > 0$, 
if $X$ is non-lattice then for all $x$ with $|x| \leq c \sqrt{n}$, 
\[
\p{x \leq S_n < x+h} = (1+o(1))\frac{h\cdot e^{-x^2/(2n\E{X^2})}}{\sqrt{2\pi\E{X^2}n}},
\]
and if $X$ is lattice then for all $x\in \mathbb{L}_X$ with $|x| \leq c \sqrt{n}$, 
\[
\p{S_n = x} = (1+o(1))\frac{e^{-x^2/(2n\E{X^2})}}{\sqrt{2\pi\E{X^2}n}}.
\]
In both cases, $o(1) \rightarrow 0$ as $n \rightarrow \infty$ uniformly over all $x$ in the allowed range. 
\end{thm}
\begin{proof}[Proof of Theorem \ref{thm:gbt_new}]
	We first remark that when $X$ is a lattice random variable with period $d$, if $A=1/d$ then $[k,k+A)$ contains precisely 
	one element from the lattice of $X$.  To shorten the formulas in the proof, we assume that $X$ is indeed lattice, 
	that $A=1/d$, and that $k$ is in the lattice of $X$, so that $k \leq S_n < k+A$ if and only if $S_n=k$. The proof of the 
	more general formulation requires only a line-by-line rewriting of what follows.
	
	Throughout this proof, $c_1,c_2$, and $c_3$ are the constants from Lemma \ref{lem:pem_per} and $\epsilon$ is the constant from Corollary \ref{cor:pem_per}. 
	We first prove the upper bound. We assume that $n \geq 4$. 
	Let $S^r$ be the random walk with $S^r_0=0$ and, for 
	$i$ with $0 \leq i < n$, with $S_r^{i+1}=S^r_i - X_{n-i}$. Define $T_0$ as in Lemma \ref{lem:pem_per} 
	and let $T^r_{k}$ be the minimum of $n$ and the first time $t$ that $S^r_t \leq -k$. 
	
	In order that $S_n = k$ and $S_i > 0$ for all $0 < i < n$, it is necessary that 
	\begin{itemize}
		\item $T_0 > \lfloor n/4 \rfloor$, 
		\item $T^r_{k+A} > \lfloor n/4 \rfloor$, and 
		\item $S_n=k$.
	\end{itemize}
	By the independence of disjoint sections of the random walk and by Lemma \ref{lem:pem_per} (iii), we have 
	\begin{eqnarray}
		\p{T_0 > \lfloor n/4 \rfloor, T^r_{k} > \lfloor n/4 \rfloor} & = & \p{T_0 > \lfloor n/4 \rfloor}\p{T^r_{k} > \lfloor n/4 \rfloor}\nonumber\\
				& \leq & \frac{c_3 \max\{k,1\}}{(\sqrt{\lfloor n/4\rfloor})^2} < \frac{8 c_3 \max\{k,1\}}{n}. \label{eq:gbt_new1}
	\end{eqnarray}
	We next rewrite the condition $S_n =k$ as 
		\[S_{\lceil 3n/4\rceil} - S_{\lfloor n/4\rfloor} = k - S_{\lfloor n/4\rfloor} - S_{\lfloor n/4\rfloor}^r.\] 
	Since $S_{\lceil 3n/4\rceil} - S_{\lfloor n/4\rfloor}$ is independent of $S_{\lfloor n/4\rfloor}$, of $S_{\lfloor n/4\rfloor}^r$, and of the events $\{T_0 > \lfloor n/4 \rfloor\}$ and $\{T^r_{k+A} > \lfloor n/4 \rfloor\}$, we have 
	\begin{equation*}
		\probC{S_n =k}{T_0 > \lfloor n/4 \rfloor, T^r_{k} > \lfloor n/4 \rfloor} \leq  \sup_{r\in \R} \p{S_{\lceil 3n/4\rceil} - S_{\lfloor n/4\rfloor} =r}.
	\end{equation*}
	By Theorem \ref{thm:spread}, we thus have 
	\begin{equation}\label{eq:gbt_new2}
	\probC{S_n =k}{T_0 > \lfloor n/4 \rfloor, T^r_{k} > \lfloor n/4 \rfloor} \leq \frac{c}{\sqrt{\lceil 3n/4\rceil - \lfloor n/4\rfloor}} < \frac{2 c}{\sqrt{n}},
	\end{equation}
	where $c$ is the constant from Theorem \ref{thm:spread}. Combining (\ref{eq:gbt_new1}) and (\ref{eq:gbt_new2}) proves the upper bound. 
	
	Fix $\alpha$ with $0 <\alpha < 1/2$ so that $4(1-2\alpha)\E{X^2}<\epsilon^2$, where $\epsilon$ is the constant from Corollary \ref{cor:pem_per}. 
	In order that $S_n =k$ and $S_i > 0$ for all $0 < i < n$, it is sufficient that the following events occur:
	\begin{itemize}
		\item[$E_1$:] $T_0 > \lfloor \alpha n \rfloor$ and $\epsilon\sqrt{n} \leq S_{ \lfloor \alpha n \rfloor} \leq \sqrt{6c_2n}$,
		\item[$E_2$:] $T_{k}^r > \lfloor \alpha n \rfloor$ and $\epsilon\sqrt{n} \leq S_{ \lfloor \alpha n \rfloor}^r \leq \sqrt{6c_2n}$,
		\item[$E_3$:] $\min_{\lfloor \alpha n \rfloor \leq i \leq \lceil (1-\alpha) n \rceil} S_i-S_{\lfloor \alpha n \rfloor} > -S_{\lfloor \alpha n \rfloor}$, and 
		\item[$E_4$:] $S_n =k$. 
	\end{itemize}
	By the independence of disjoint sections of the random walk, we therefore have 
	\begin{equation}\label{eq:gbt_new2.5}
		\p{S_n=k,S_i > 0~\forall~0 < i < n} \geq \p{E_1}\cdot \p{E_2} \cdot \probC{E_3,E_4}{E_1,E_2}.
	\end{equation}
	By Lemma \ref{lem:pem_per} (i), we have $\p{T_0 > \lfloor \alpha n \rfloor} \geq c_1/\sqrt{n}$, so by the FKG inequality we see that 
	\begin{equation}\label{eq:gbt_new3}
		\p{T_0 > \lfloor \alpha n \rfloor, S_{\lfloor \alpha n \rfloor} \geq \epsilon\sqrt{n} } \geq \frac{c_1}{3\sqrt{n}}.
	\end{equation}
	By applying Corollary \ref{cor:pem_per} and Markov's inequality, we also have  
	\begin{equation}\label{eq:gbt_new4}
		\probC{S_{ \lfloor \alpha n \rfloor} > \sqrt{6c_2n}}{T_0 > \lfloor \alpha n \rfloor, S_{\lfloor \alpha n \rfloor} \geq \epsilon\sqrt{n}} \leq \frac{\E{S_n^2~\vert~T_0 > \lfloor \alpha n\rfloor, S_{\alpha n} \geq \epsilon n}}{6c_2n} 
			\leq \frac{1}{2},
	\end{equation}
	and (\ref{eq:gbt_new3}) and (\ref{eq:gbt_new4}) together imply that 
	\begin{equation}\label{eq:gbt_new5}
	\p{E_1} \geq \frac{c_1}{6\sqrt{n}}.
	\end{equation}
	An identical argument shows that $\p{E_2} \geq c_1\cdot \max\{k,1\}/(6\sqrt{n})$. By (\ref{eq:gbt_new2.5}), 
	to prove the lower bound it thus suffices to show that 
	there is $\gamma > 0$ not depending on $n$ or on our choice of $k$ and such that $\probC{E_3,E_4}{E_1,E_2} \geq \gamma/\sqrt{n}$; we now turn to 
	establishing such a bound.
	
	Let $m =\lceil (1-\alpha)n\rceil-\lfloor \alpha n \rfloor$. For $1 \leq i \leq m$, let 
	\[
	L_i = S_{\lfloor \alpha n \rfloor+i}-S_{\lfloor \alpha n \rfloor}, \hspace{0.2cm} \mbox{and let} \hspace{0.2cm} R_i=S_{\lceil (1-\alpha)n\rceil-i}-S_{\lceil (1-\alpha)n\rceil},
	\] 
	so in particular $L_m=-R_m$. 
	Next, rewrite the event $E_4$ as 
	\[
		L_m = k - S_{\lfloor \alpha n\rfloor}-S^r_{\lfloor \alpha n\rfloor}.
	\]
	By the independence of disjoint sections of the random walk, we then have 
	\begin{eqnarray}\label{eq:gbt_new6}
	\probC{E_3,E_4}{E_1,E_2} 	&\geq	& \inf_{p,q\in [\epsilon\sqrt{n},\sqrt{6c_2n}]\cap \mathbb{L}_X} \left\{\mathbf{P}\big\{L_m= -p+(k+q), \min_{1 \leq i \leq m} L_i > -p\big\}\right\}.
	\end{eqnarray}
	For $p$ and $q$ as in (\ref{eq:gbt_new6}), we define the following shorthand:
	\begin{itemize}
		\item $A_{p,q}$ is the event that $L_m=-p+(k+q)$, and
		\item $B_p$ is the event that $\min_{1 \leq i \leq m} L_i > -p$. 
	\end{itemize}
	We bound $\p{A_{p,q},B_p}$ by first writing 
	\begin{equation}\label{eq:gbt_new7}
		\p{A_{p,q},B_p} \geq \p{A_{p,q}}-\p{A_{p,q},\overline{B_p}}.
	\end{equation}
	By Theorem \ref{thm:stone}, for all $k$ with $0 \leq k \leq \sqrt{n}$ and all $p,q \in [\epsilon\sqrt{n},\sqrt{6c_2n}]\cap\mathbb{L}_X$, 
	\begin{equation}\label{eq:gbt_new8}
		\p{A_{p,q}} = (1+o(1))\frac{e^{-(k+q-p)^2/(2\sigma^2 m)}}{\sqrt{2\pi \sigma^2 m}},
	\end{equation}
	where $o(1)\rightarrow 0$ as $n \rightarrow \infty$, uniformly over all $k$, $p$, and $q$ as above. 
	To bound $\p{A_{p,q},\overline{B_p}}$ from above, we first further divide the events $A_{p,q}$, $B_p$. Let $m'=\lfloor n/2\rfloor - \lfloor \alpha n \rfloor$. Observe that 
	$A_{p,q}$ occurs if and only if $R_m = -(k+q)+p$. Similarly, if $A_{p,q}$ occurs, then for $\overline{B_p}$ to occur one of the following events must occur: either
	\begin{enumerate}
		\item $\min_{1 \leq i \leq m'} L_i \leq -p$ (we call this event $C_p$); or 
		\item $\min_{1 \leq i \leq m-m'} R_i \leq -(k+q)$ (we call this event $D_q$). 
	\end{enumerate}
	Thus, 
	\begin{equation}\label{eq:gbt_new9}
		\p{A_{p,q},\overline{B_p}} \leq \p{C_p, L_m=-p+(k+q)} + \p{D_q,R_m=-(k+q)+p}.
	\end{equation}
	By Kolmogorov's inequality and our choice of $\alpha$, 
	\begin{equation}\label{eq:gbt_new10}
		\p{C_p} \leq \frac{\E{L_m^2}}{p^2}= \frac{\E{X^2}\cdot mm}{p^2} \leq \frac{\E{X^2}((1-2\alpha)n+1)}{\epsilon^2 n} < \frac{1}{4},
	\end{equation}
	for all $n$ sufficiently large. 
	Furthermore, since for any $i$ with $1 \leq i \leq m'$, $L_m-L_i$ is a sum of $m-i \geq m-m' \geq \lceil (1-\alpha)n \rceil - \lfloor n/2 \rfloor$ copies of $X$, 
	by the independence of disjoint sections of the random walk and by Theorem \ref{thm:stone}, 
	\begin{eqnarray}\label{eq:gbt_new11}
		\probC{L_m = -p+(k+q)}{L_i \leq -p} 	&=& (1+o(1))\frac{e^{-(k+q)^2/(2\sigma^2(m-i))}}{\sqrt{2\pi\sigma^2(m-i)}} \nonumber\\
									&\leq& (1+o(1)) \frac{e^{-(k+q-p)^2/(2\sigma^2 m)}}{\sqrt{2\pi \sigma^2 m}},
	\end{eqnarray}
	where $o(1)\rightarrow 0$ as $n \rightarrow \infty$, uniformly over all $i,k,p$, and $q$ as above. It follows by Bayes' formula that 
	\begin{eqnarray}\label{eq:gbt_new12}
		\p{C_p,L_m=-p+(k+q)} & \leq & \p{C_p}\cdot \max_{1 \leq i \leq m} \probC{L_m = -p+(k+q)}{L_i \leq -p} \nonumber\\
						    & \leq & (1+o(1)) \frac{e^{-(k+q-p)^2/(2\sigma^2 m)}}{4\sqrt{2\pi \sigma^2 m}}.
	\end{eqnarray}
	A similar calculation shows that 
	\begin{equation}\label{eq:gbt_new13}
		\p{D_q,R_m = -(k+q)+p} \leq (1+o(1)) \frac{e^{-(k+q-p)^2/(2\sigma^2 m)}}{4\sqrt{2\pi \sigma^2 m}},
	\end{equation}
	and combining (\ref{eq:gbt_new7}), (\ref{eq:gbt_new8}), (\ref{eq:gbt_new12}) and (\ref{eq:gbt_new13}), since $m$ and $n$ have the same order we see that 
	\begin{equation}\label{eq:gbt_new14}
		\p{A_{p,q},B_p} \geq (1+o(1))\frac{e^{-(k+q-p)^2/(2\sigma^2 m}}{2\sqrt{2\pi \sigma^2 m}} \geq \frac{\gamma}{\sqrt{n}},
	\end{equation}
	for all $k$, $p$ and $q$ in the allowed ranges, for some $\gamma$ not depending on $k$, $p$, $q$, or $n$. 
	Together (\ref{eq:gbt_new6}) and (\ref{eq:gbt_new14}) establish the required lower bound on $\probC{E_3,E_4}{E_1,E_2}$ and complete the proof.
\end{proof}

\section{Counterexamples}\label{sec:counterexamples}

In Section \ref{sec:normal} we exhibit a random walk $S$ with mean zero step, finite variance step size $X$ and show that 
with $k=n$, $\p{S_n>0~\forall~0<i<n\vert S_n=k}$ is not $\Theta(k/n)=\Theta(1)$ as the ballot theorem would suggest. 
This example may be easily modified to show that we can not in general expect a result of the 
form $\p{S_n>0~\forall~0<i<n\vert S_n=k}=\Theta(k/n)$ 
for any $k=\Omega(\sqrt{n}\log n)$, and we believe that with a little effort it should be possible to 
make this approach work for any $k=\omega(\sqrt{n})$. 

In Section \ref{sec:other} we exhibit a random walk $S$ with step size $X$ for which $X$ is an integer random variable with period $1$, 
and for which $\e{X}=0$, $\E{X^{3/2-\epsilon}}<\infty$ for all $\epsilon>0$, but such that $\mathbf{P}\{S_n>0~\forall~ 0<i<n~\vert~S_n=\sqrt{n}\}$
is not $\Theta(n^{-1/2})$; the same idea can be easily modified to yield a random variable $X$ 
with $\e{X^{\alpha}}<\infty$ for any fixed $\alpha < 2$ and for which the ballot theorem 
can be seen to fail even in the range $S_n=O(\sqrt{n})$. 

The ideas behind our examples are most easily explained from the multiset mentioned in the introduction. 
The ``underlying multiset'' $\ssf$ for our example showing that the condition $k=O(n)$ is necessary 
consists of $(n-1)/2$ elements of value $1$, the same number of elements of value $-1$'s, and a single element of value $n$. 
The elements of $\ssf$ sum to $n$, and in order for all partial sums to stay positive, it 
is necessary and sufficient that the partial sums not containing the element $n$ stay positive 
(as all partial sums containing $n$ are certainly positive). 
Denoting the elements of $\ssf$ by $x_1,\ldots,x_n$ and letting $\sigma$ be a uniformly 
random permutation of $\{1,\ldots,n\}$, the index $i$ for which $x_{\sigma(i)}=n$ 
is uniform among $\{1,\ldots,n\}$. Letting $S_{\sigma(i)}=\sum_{j=1}^i x_{\sigma(j)}$ for $0<i\leq n$, 
we may thus write
\bea
\p{S_{\sigma(i)}>0~\forall~0<i<n} &=& \frac{1}{n}\cdot\sum_{j=1}^n \p{S_{\sigma(i)}>0~\forall~0<i<n\vert x_{\sigma(j)}=n} \eal
                        &=& \frac{1}{n}\cdot\sum_{j=1}^n \p{S_{\sigma(i)}>0~\forall~0<i<j\vert x_{\sigma(j)}=n}. \label{eq:set_example}
\eea
For a symmetric simple random walk $S'$, it is well-known (see, e.g., \cite{feller68intro1}, Lemma III.3.1) that for integers $j>0$, 
$\p{S_i'>0~\forall~0<i<j}=O(j^{-1/2})$. Making the hopefully plausible leap of faith that the same bound holds for 
$\p{S_{\sigma(i)}>0~\forall~0<i<j\vert x_{\sigma(j)=n}}$, (\ref{eq:set_example}) then yields that 
\[\p{S_{\sigma(i)}>0~\forall~0<i<n}=\frac{1}{n}\cdot\sum_{j=1}^nO\left(\frac{1}{\sqrt{j}}\right) = O\left(\frac{1}{\sqrt{n}}\right),\]
not $\Theta(1)$ as the ballot theorem would suggest. 
Turning this intuition into the example of Section \ref{sec:normal} is a matter of finding a random walk $S$ with 
mean zero step size $X\in \ssd$ for which, given that $S_n=n$, the set $\{X_1,\ldots,X_n\}$ very likely looks 
like the set $\ssf$ above, i.e., there is a single index $i$ for which $X_i=n$, and for all other $j$, $X_j$ is ``small''. 

For the example showing that the condition $\E{X^2} < \infty$, ``underlying multiset'' we are thinking of 
consists of roughly $(n-n^{1/4})/2$ elements of value $1$, the same number of elements of value $+1$, 
$(n^{1/4}+1)/2$ elements of value $\sqrt{n}$, and $(n^{1/4}-1)/2$ elements of value $-\sqrt{n}$. 
These elements sum to $\sqrt{n}$. 

For all partial sums in a uniformly random permutation of this multiset to stay positive, 
it is necessary that the partial sums stay positive until an element of value $\sqrt{n}$ is sampled -- 
this should occur after about $n^{3/4}$ elements have been sampled, so the intuition given by a 
symmetric simple random walk suggests that the partial sums stay positive until this 
time with probability of order $O(n^{-(3/4)\cdot(1/2)})=O(n^{-3/8})$. 

In order that the partial sums stay positive, it is also essentially necessary that the 
``sub-random walk'' consisting of the partial sums {\em of only elements of absolute value $\sqrt{n}$} 
stays positive -- for if this ``sub-random walk'' becomes extremely negative 
then it is very unlikely that the original partial sums stay positive. 
Dividing through by $n^{1/2}$ we can view this ``sub-random walk'' as a symmetric simple random walk 
$S'$ of length $n^{1/4}$, conditioned 
on having $S_{n^{1/4}}'=1$. By the ballot theorem, the probability such a random walk stays positive is 
$O(1/n^{1/4})$. Combining the bounds of the this paragraph and the previous 
paragraph as though the two events were independent (which, though clearly false, gives the correct 
intuition) suggests that the original partial sums should stay positive with probability $O(1/n^{3/8+1/4})=O(1/n^{5/8})$, 
not $\Theta(n^{-1/2})$ as the ballot theorem would suggest.  

Before we turn to the details of these examples, we first spend a moment gathering two easy lemmas that we will use in the course of 
their explanation. 

\subsection{Two Easy Lemmas}
The first lemma bounds the probability that a random walk stays above zero until some time $m$, and 
is a simplification of  \cite{feller68intro2}, 
Theorem XII.7.12a.
\begin{lem}\label{lem:easy1}
Given a random walk $S$ with step size $X$, if $X$ is symmetric then for integers 
$m > 0$, 
\[\p{S_i>0~\forall~0<i\leq m}=\Theta\left(\frac{1}{\sqrt{m}}\right)\]
\end{lem}
The second lemma is an easy extension of classical \cite{chernoff52measure} bounds to a setting in which the number 
of terms in the binomial is random. The classical Chernoff bounds 
(see, e.g., (2.5) and (2.6) in \cite{janson00random} for a modern reference), state: 
given a binomial random variable $\Bin(n,p)$ with mean $\mu=np$, 
for all $c > 0$, 
\bea
\p{\Bin(n,p) > (1+c)\mu} & \leq & e^{-c^2\mu/(2(1+c/3))}\label{eq:chernoff_u}\\
\p{\Bin(n,p) < (1-c)\mu} & \leq & e^{-c^2\mu/2}\label{eq:chernoff_l}
\eea
The following lemma follows from the Chernoff bounds by straightforward applications 
of Bayes' formula:
\begin{lem}\label{lem:chernoff_rand}
        Let $m$ be a positive integer, let $0 < q < 1$, and let $U$ be distributed as $\Bin(m,q)$. 
        Let $v$ be a positive real number and let $V_1,V_2,\ldots$ be i.i.d.~random variables 
        taking values $\pm v$, each with probability $1/2$. Finally, let $Y=V_1+\ldots+V_U$. 
        The for all $t>0$, 
        \bea
        \p{Y > t} & \leq &\exp\left\{-\frac{t^2}{8mq+\frac{4tv}{3}}\right\}+\exp\left\{-\frac{mq}{3}\right\},\quad\mbox{and}\eal
        \p{Y < -t} & \leq & \exp\left\{-\frac{t^2}{8mq}\right\}+\exp\left\{-\frac{mq}{3}\right\}.\nonumber
        \eea
\end{lem}
\begin{proof}
        As $\e{U}=mq$, by (\ref{eq:chernoff_u}) we have 
        \bea
                \p{Y > t} & \leq & \sum_{u=1}^{\lfloor 2mq \rfloor} \p{Y>t\vert U=u}\p{U=u} + \p{U>2mq}\eal
                          & \leq & \sup_{u \leq 2mq} \p{Y>t \vert U=u} + \p{U>2mq} \eal
                                  & \leq & \sup_{u \leq 2mq} \p{Y>t \vert U=u} + e^{-\frac{mq}{3}}\label{eq:ch_ra_1}
        \eea
        Given that $U=u$, $(Y+uv)/2v$, which we denote $Y'$, is distributed as $\Bin(u,1/2)$. 
        Furthermore, in this case $Y > t$ if and only if 
        \[ Y' > \frac{u}{2}+\frac{t}{2v} = \frac{u}{2}\left(1+\frac{t}{uv}\right).\]
        It follows by (\ref{eq:chernoff_u}) that 
        \bea
        \p{Y>t\vert U=u} & \leq & \exp\left\{-\frac{u}{2}\left(\frac{t}{uv}\right)^2\left(2+\frac{2t}{3uv}\right)^{-1}\right\} \eal
                                         & =    & \exp\left\{-\frac{t^2}{4uv^2+\frac{4tv}{3}}\right\} \eal 
										 & < 	& \exp\left\{-\frac{t^2}{8mq+\frac{4tv}{3}}\right\},\nonumber
        \eea
        where in the last inequality we use that $u \leq 2mq$ and $v \geq 1$. As $u$ was arbitrary, combining this inequality with 
        (\ref{eq:ch_ra_1}) yields the desired bound on $\p{Y>t}$. The bound on $\p{Y<-t}$ is proved identically.
\end{proof}

\subsection{Optimality for normal random walks}\label{sec:normal} 
Let $f$ be the tower function: $f(0)=1$ and $f(k+1)=2^{f(k)}$ for integers $k \geq 0$. We define a random variable $X$ as follows: 
\[X =   \begin{cases}
                        \pm 1, & \mbox{each with probability }\frac{1}{2}\cdot\left(1-\sum_{i=0}^{\infty}\frac{1}{f(k)^4}\right) \\
                        \pm f(k), & \mbox{each with probability }\frac{1}{2f(k)^4},\mbox{ for }k=1,2,\ldots \\
                \end{cases}
\]
and let $S$ be a random walk with steps distributed as $X$. 
Clearly $\e{X}=0$, and it is easily checked that $\E{X^2} < 2$. 
We will show that when $n=f(k)$ for positive integers $k$, 
$\p{S_i > 0~\forall 0 < i < n \vert S_n=n}$ is $O(1/\sqrt{n})$, so in 
particular, for such values of $n$ this probability is not 
$\Theta(1)$ as the ballot theorem would suggest. 

For $i \geq 0$ we let $N_i$ be the number of times $t$ with $1 \leq t \leq n$ that $|X_t|=f(i)$; 
clearly $\sum_{i=0}^{\infty} N_i=n$. 
For $S_n=n$ to occur, it suffices that for some $t$ with $1 \leq t \leq n$, 
$X_t=n$ and $S_n-X_t=0$; therefore
\[\p{S_n=n} \geq \sum_{t=1}^n\p{X_t=n,S_n-X_t=0} \geq \sum_{t=1}^n\frac{\p{S_n-X_t=0}}{2n^4}.\]
As $\e{X^2}<\infty$ and, for all $1 \leq t \leq n$, $S_n-X_t$ is simply a sum of $n-t$ independent copies 
of $X$, by Theorem \ref{thm:stone} we know that $\p{S_n-X_t=0}=\Theta(n^{-1/2})$ uniformly over 
$1 \leq i \leq n$.  
It follows that 
\begin{equation}\label{eq:prob_snequalsn}
        \p{S_n=n}=\Omega\left(\frac{1}{n^{7/2}}\right).
\end{equation}
We next bound the probability that $S_n=n$ and $S_t > 0$ for all $0 < t < n$; 
we denote this conjunction of events $E$. Our aim is to show that $\p{E}=O(n^{-4})$, 
which together with (\ref{eq:prob_snequalsn}) will establish our claim that 
$\p{S_i > 0~\forall~0 < i < n \vert S_n=n}$ is $O(1/\sqrt{n})$. Recalling that 
$f(k)=n$, we write 
\bea
\p{E} & = & \p{E,N_k=0}+\p{E,N_k\geq 1}.\label{eq:split_four}
\eea
It is easy to show using Chernoff bounds that 
$\p{E,N_k=0}=O(n^{-6})$ (we postpone this step for the moment). 
From this fact and from (\ref{eq:split_four}), we therefore have 
\[\p{E} = \p{E,N_k\geq1} + O\left(n^{-6}\right),\]
from which it will follow that $\p{E}=O(n^{-4})$ if we can show that 
$\p{E,N_k\geq1}=O(n^{-4})$. We first do so, then justify our assertion that 
$\p{E, N_k=0}=O(n^{-6})$.

For $E$ and $\{N_k\geq1\}$ to occur, one of the following events must occur.
\begin{itemize}
        \item For some $t$ with $1 \leq t \leq \lfloor n/2 \rfloor$, $X_t=\pm n$, $S_i>0$ for all $0 < i < t$, 
        and $S_n=n$. We denote these events $B_t$, for $1 \leq t \leq \lfloor n/2 \rfloor$, 
        and remark that they are {\em not} disjoint. 
        \item $S_i > 0$ for all $0 < i < \lfloor n/2 \rfloor$, and for some $t$ with $\lfloor n/2 \rfloor < t \leq n$, 
        $X_t=\pm n$ and $S_n=n$. We denote these events 
        $D_t$ for $\lfloor n/2 \rfloor \leq t \leq n$; again, they are not disjoint. 
\end{itemize}
We first bound the probabilities of the events $B_t$, $1 \leq t \leq \lfloor n/2 \rfloor$. Fix some $t$ 
in this range -- by Lemma \ref{lem:easy1}, the probability that $S_i > 0$ for all 
$0 < i < t$ is $O(t^{-1/2})$ uniformly in $t$. By definition, $\p{X_t=\pm n}=n^{-4}$, so by the strong Markov property, 
\[\p{S_i > 0~\forall~0 < i < t,X_t=\pm n}=O\left(\frac{1}{n^4\sqrt{t}}\right),\]
still uniformly in $t$. Furthermore, by Theorem \ref{thm:spread} we have that 
\begin{equation}\label{eq:sn_spread}
        \sup_{r}\p{S_n-S_t=r}=O\left(\frac{1}{\sqrt{n-t}}\right)=O\left(\frac{1}{\sqrt{n}}\right).
\end{equation}
As $S_n-S_t$ and $n-S_t$ are independent, it follows from (\ref{eq:sn_spread}) that $\p{S_n-S_t=n-S_{t}}=O(n^{-1/2})$, 
and by another application of the strong Markov property we have that 
\[\p{B_t} = \p{S_i > 0~\forall 0 < i < t,X_t=\pm n,S_n-S_t=n-S_t}=O\left(\frac{1}{n^{9/2}\sqrt{t}}\right).\] 
Thus,
\bea
\p{\bigcup_{t=1}^{\lfloor n/2\rfloor}B_t}&\leq&\sum_{t=1}^{\lfloor n/2\rfloor}\p{B_t}\leq \sum_{t=1}^{\lfloor n/2\rfloor} 
O\left(\frac{1}{n^{9/2}\sqrt{t}}\right) = O\left(\frac{1}{n^4}\right).\label{eq:bt_union_bound}
\eea
We next bound the probabilities of the events $D_t$, $\lfloor n/2\rfloor < t \leq n$. By an argument just as above, 
we have that 
\[\p{S_i > 0~\forall 0<i<\lfloor n/2\rfloor,X_t=\pm n}=O\left(\frac{1}{n^{9/2}}\right).\]
Also just as above, since $S_n-S_{\lfloor n_2\rfloor}-X_t$ and $n-S_{\lfloor n_2\rfloor}-X_t$ are independent, 
\[\p{S_n-S_{\lfloor n_2\rfloor}-X_t=n-S_{\lfloor n_2\rfloor}-X_t}=O\left(\frac{1}{\sqrt{n}}\right).\]
By the independence of disjoint sections of the random walk we therefore have that $\p{D_t}=O(n^{-5})$, 
and so 
\begin{equation}\label{eq:dt_union_bound}
        \p{\bigcup_{t=\lfloor n/2\rfloor+1}^n D_t}=O\left(\frac{1}{n^{4}}\right). 
\end{equation}
As $E\cap\{N_k\geq1\}$ is contained in $\bigcup_{t=1}^{\lfloor n/2\rfloor}B_t\cup \bigcup_{t=\lfloor n/2\rfloor+1}^n D_t$, 
it follows from (\ref{eq:bt_union_bound}) and (\ref{eq:dt_union_bound}) that $\p{E,N_k \geq 1}=O(n^{-4})$ as claimed. 

We now turn our attention to proving that $\p{E,N_k=0}=O(n^{-6})$. We will in fact show that 
\begin{equation}\label{eq:ns_bound}
        \p{N_k=0,|S_n| \geq 8k\sqrt{n\log n}}=O\left(\frac{1}{n^6}\right),
\end{equation}
which implies the desired bound. 
We recall that $N_i$ is the number of times $t$ with $1 \leq t \leq n$ that $X_t=f(i)$. 
For $0 \leq i < k$, let $S_{n,i}=\sum_{\{1 \leq j \leq n:|X_j|=f(i)\}}X_j$. 
If $N_k=0$ then either $S_n=\sum_{i=1}^{k-1}S_{n,i}$ 
or $N_j>0$ for some $j >k$; 
thus, 
\begin{equation}\label{eq:contained}
        \{N_k=0,|S_n|>8k\sqrt{n\log n}\}\subseteq\left\{\sum_{j=k+1}^{\infty}N_j>0\right\}\cup\bigcup_{i=1}^{k-1}\{|S_{n,i}|>8\sqrt{n\log n}\}.
\end{equation}
For any $1 \leq i < k$, 
$S_{n,i}$ is the sum of $N_i$ i.i.d. random variables taking values $\pm g(i)$, each with probability $1/2$, 
and $N_i$ is distributed as $\Bin(n,g(i)^{-4})$. By Lemma \ref{lem:chernoff_rand}, therefore, 
\bea
\p{|S_{n,i}|>8\sqrt{n\log n}} & \leq & 2\exp\left\{\frac{-64n\log n}{\frac{8n}{g(i)^2}+\frac{32\sqrt{n\log n}g(i)}{3}}\right\}+2\exp\left\{\frac{-n}{3g(i)^4}\right\}.\nonumber
\eea
Since $g(i)\leq \log n$, presuming $n$ is large enough that $32\sqrt{n\log n}g(i)<n$ we thus have 
\bea
\p{|S_{n,i}|>8\sqrt{n\log n}} & \leq & 2\exp\left\{\frac{-64\log n}{9}\right\}+2\exp\left\{\frac{-n}{3\log^4 n}\right\}=O\left(\frac{1}{n^7}\right).\label{eq:tilarge_bound}
\eea
Furthermore, it follows directly from the definition of $X$ that $\p{\sum_{j=k+1}^{\infty}N_j>0}=o(2^{-n})$. 
Applying this fact, (\ref{eq:tilarge_bound}), and (\ref{eq:contained}), 
it follows immediately that 
\[
\p{N_k=0,|S_n|>8k\sqrt{n\log n}} = o(2^{-n}) + O\left(\frac{k}{n^7}\right) = O\left(\frac{1}{n^6}\right)
\]
as claimed.

\subsection{Optimality for other random walks}\label{sec:other}
Let $g:\N\rightarrow \N$ be an rapidly increasing integer-valued function with $g(0)=1$; 
in particular, we choose $g$ such that $g \geq f$ where $f$ is the tower function seen in the previous section. 
\[X =   \begin{cases}
                        \pm 1, & \mbox{each with probability }\frac{1}{2}\left(1-\sum_{i=0}^{\infty}\frac{1}{g(k)^{3/2}}\right) \\
                        \pm g(k), & \mbox{each with probability }\frac{1}{2g(k)^{3/2}},\mbox{ for }k=1,2,\ldots \\
                \end{cases}
\]
and let $S$ be a random walk with steps distributed as $X$. 
Clearly $\e{X}=0$. 
For integers $i>0$, let $\eventn{Pos}_i$ be the event that $S_j > 0$ for all $0 < j \leq i$; 
we also let $\eventn{Pos}_0$ be some event of probability $1$ as it will simplify later equations. 
We will show that when $n=g(k)^2$ for positive integers $k$, 
\begin{equation}\label{eq:ce2_sn}
        \p{S_n=\sqrt{n}} = \Omega\left(\frac{1}{\sqrt{\log n}n^{5/8}}\right), 
\end{equation}
and that 
\begin{equation}\label{eq:ce2_snpos}
        \p{S_n=\sqrt{n},\eventn{Pos}_n} = O\left(\frac{\log^{13/2} n}{n^{5/4}}\right), 
\end{equation}
from which it follows by Bayes' formula that for such values of $n$
$\p{\eventn{Pos}_n \vert S_n=\sqrt{n}}$ is $O(\log^7n/n^{5/8})$, not 
$\Theta(1/\sqrt{n})$ as the ballot theorem would suggest. We now prove 
(\ref{eq:ce2_sn}) and (\ref{eq:ce2_snpos}). In what follows we presume, 
to avoid cumbersome floors and ceilings, that $g(k)=\sqrt{n}$ has been 
chosen so that $\sqrt{\log n}$ and $n^{1/8}$ are both integers. 

For $j=1,2,\ldots$ and $i=0,1,\ldots$, we let $\xf_{j,i}$ be the random set 
$\{X_m:1\leq m\leq j,|X_m|=g(i)\}$, let $N_{j,i}=|\xf_{j,i}|$, and let 
$S_{j,i}=\sum_{X_m \in \xf_{j,i}}X_m$. For all $j=1,2,\ldots$, 
the sets $\xf_{j,0},\xf_{j,1},\ldots$ partition $\{X_1,\ldots,X_j\}$. 
For an integer $k \geq 0$, the {\em $k$-truncated} random walk $S^{(k)}$ is given by 
\[
S^{(k)}_j=\sum_{i=1}^{k}S_{j,i} = \sum_{i=1}^jX_i\I{|X_i|\leq g(k)},
\] 
for $j=1,2,\ldots$. We remark that for any $n$ and any set $\xf \subseteq \{X_1,\ldots,X_n\}$, 
conditional upon the event that $\bigcup_{i=1}^k \xf_{n,i} = \xf$, $S_n^{(k)}$ is simply a sum 
of $|\xf|$ i.i.d.~{\em bounded} random variables with variance at most $g(k)^2$. 
In particular, this implies that {\em after} such conditioning, $S_n^{(k)}$ obeys a local 
central limit theorem around $0$. The key consequence of this fact (for our purposes), 
is that we may choose $g$ to grow fast enough that there exists a small constant $\epsilon>0$ such that:
\bea
        \forall~k \geq 0,~\forall~n'\geq n \geq g(k+1),~\forall~\xf\subset \{X_1,\ldots,X_{n'}\}\mbox{ s.t. }|\xf|=n, & & \eal
        \Cprob{S^{(k)}_{n'}=0}{\{X_i:1\leq i \leq n',|X_i|\leq g(k)\}=\xf}\geq \frac{\epsilon}{g(k)\sqrt{n}};\quad & & \label{eq:normapprox}
\eea
such a constant is guaranteed to exist by Theorem \ref{thm:stone} and our above observation about the 
conditional distribution of $S^{(k)}_n$.

Fix some integer $k\geq 1$ and let $n=g(k)^2$. 
We remark that $\e{N_{n,k}}=n\p{X=g(k)}=n^{1/4}/2$. 
For $S_n=\sqrt{n}$ to occur, it suffices that $S_{n,k}=g(k)=\sqrt{n}$ and that 
$S_n-S_{n,k}=0$. For any subset $\sk$ of $\ssf=\{X_1,\ldots,X_n\}$, $S_{n,k}$ and $S_n-S_{n,k}$ are 
conditionally independent given that $\xf_k=\sk$. 
Letting $\zf$ be the set of subsets of $\ssf$ of odd parity and of size at most $2n^{1/4}$, 
we then have
\bea
\p{S_n=\sqrt{n}} & \geq & \p{S_{n,k}=g(k),S_n-S_{n,k}=0} \eal
                                 & \geq & \p{S_{n,k}=g(k),S_n-S_{n,k}=0,N_{n,i} \leq 2n^{1/4},N_{n,i}\mbox{ odd}} \eal
                                 & \geq & \left(\inf_{\sk\in \zf}\p{S_{n,k}=g(k),S_n-S_{n,k}=0\vert \xf_k=\sk}\right)\cdot\p{N_{n,i} \leq 2n^{1/4},N_{n,i}\mbox{ odd}}\eal
                                 & =    & \left(\inf_{\sk\in \zf}\p{S_{n,k}=g(k)\vert \xf_k=\sk}\p{S_n-S_{n,k}=0\vert \xf_k=\sk}\right)\cdot\left(\frac{1}{2}-o(1)\right),\eal
                                & & \label{eq:twoterms}
\eea
by the aforementioned independence and a Chernoff bound. 

To bound this last formula from below, fix an arbitrary 
element $\sk$ of $\zf$. 
Note that $S_n-S_{n,k}=S^{(k-1)}_n$ unless there is $i>k$ such that 
$\xf_i\neq \emptyset$. Since $\p{|X| > g(k)} = O(g(k+1)^{3/2}) = O(2^{-3n/4})$, it is easily seen that 
$\p{\bigcup_{i=k}^{\infty}\xf_i=\sk \vert\xf_k=\sk}=1-o(2^{-n/2})$. 
Thus, by Bayes' formula,  
\bea
\p{S_n-S_{n,k}=0\vert \xf_k=\sk}  & = & \probC{S_n-S_{n,k}=0}{\bigcup_{i=k}^{\infty}\xf_k=\xf_k=\sk}(1-o(2^{-n/2}))\eal
                        & = & \probC{S^{(k-1)}_n=0}{\bigcup_{i=k}^{\infty}\xf_k=\xf_k=\sk}(1-o(2^{-n/2})).
\eea
Letting $\xf=\{X_1,\ldots,X_n\}-\sk$, the previous equation implies that 
\begin{equation}\label{eq:sn_tk_2}
        \p{S_n-S_{n,k}=0\vert \xf_k=\sk} = \Omega\left(\probC{S^{(k-1)}_n=0}{\{X_i:1\leq i \leq n,|X_i|\leq g(k-1)\}=\xf}\right).
\end{equation}
As $|\sk|\leq 2n^{1/4}$, $|\xf|=n-|\sk|\geq n-2n^{1/4}\geq n^{1/2}=g(k)$, so applying (\ref{eq:normapprox}) in (\ref{eq:sn_tk_2}) 
yields that 
\begin{equation}\label{eq:sn_tk_3}
        \p{S_n-S_{n,k}=0\vert \xf_k=\sk}= \Omega\left(\frac{1}{g(k-1)\sqrt{n-|\sk|}}\right)=\Omega\left(\frac{1}{g(k-1)\sqrt{n}}\right).
\end{equation}
Next, for any set $\sk \in \zf$, given that $\xf_k=\sk$, it follows directly from a binomial approximation that 
$\p{S_{n,k}=g(k)\vert |\xf_k|=\sk}=\Omega(|\sk|^{-1/2})=\Omega(n^{-1/8})$. Plugging this bound and (\ref{eq:sn_tk_3}) into (\ref{eq:twoterms}) 
yields that 
\begin{equation}\label{eq:sn_58_bound}
        \p{S_n=\sqrt{n}}=\Omega\left(\frac{1}{g(k-1)n^{5/8}}\right)=\Omega\left(\frac{1}{\sqrt{\log n}n^{5/8}}\right),
\end{equation}
establishing (\ref{eq:ce2_sn}).

We next turn to our upper bound on $\p{S_n=\sqrt{n},\eventn{Pos}_n}$. 
We shall define several ways in which 
the walk $S$ can ``behave unexpectedly''. We first show that the walk is unlikely to behave 
unexpectedly; it will be fairly easy to show that given that none of the unexpected events occur, 
the probability that $\{S_n=\sqrt{n}\}$ and $\eventn{Pos}_n$ both occur is $O(\log^6 n/n^{5/4})$. 
Combining this bound with our bounds on the probability of unexpected events will yield (\ref{eq:ce2_snpos}).

We first describe and bound the probabilities of the so-called ``unexpected events''. 
Let $B$ be the event that there is $i$ with $1 \leq i \leq n$ for which $|X_i| > g(k)$. By a union bound, 
\bea
\p{B} &\leq& n\p{|X_1|>g(k)} = n\left(\sum_{i=k+1}^{\infty}\p{X_1=g(i)}\right) = n\left(\sum_{i=k+1}^{\infty}\frac{1}{g(i)^{3/2}}\right) \eal
          &=   & nO\left(\frac{1}{2^{3n/4}}\right) = o\left(\frac{1}{2^{n/2}}\right)\label{eq:bbound}
\eea
Next, let $T$ be the first time for which $X_T = g(k) = \sqrt{n}$. 
Letting $t^*=5n^{3/4}\log n$, we have 
\begin{equation}\label{eq:tgrowbound}
        \p{T>t^*} \leq \p{\bigcap_{t=1}^{t^*}\{|X_t|\neq\sqrt{n}\}} = \left(1-\frac{1}{n^{3/4}}\right)^{5n^{3/4}\log n} = O\left(\frac{1}{n^5}\right)
\end{equation}
Now, by another union bound, 
\begin{equation}\label{eq:stbbound_1}
        \p{S_{T-1}>8k\sqrt{t^*\log n},T\leq t^*,B} \leq t^*\sup_{1 \leq t \leq t^*}\p{|S_{t-1}|>8k\sqrt{t^*\log n},T=t,B}
\end{equation}
If $\{T=t\}$ and $B$ occur, then $S_{t-1}=\sum_{i=1}^{k-1}S_{t-1,i}$. Thus, by an argument just as we used to prove 
(\ref{eq:ns_bound}), we can see that 
\begin{equation}
        \p{|S_{t-1}|>8k\sqrt{t^*\log n},T=t,B} = O\left(\frac{1}{n^6}\right), 
\end{equation}
which, combined with (\ref{eq:stbbound_1}), yields that 
\begin{equation}\label{eq:stbbound_2}
        \p{|S_{T-1}|>8k\sqrt{t^*\log n},T\leq t^*,B} = O\left(\frac{t^*}{n^6}\right) = O\left(\frac{1}{n^5}\right).
\end{equation}
Finally, for $0 \leq t < n$, let $\zf_t$ be the set of subsets of $\{t+1,\ldots,n\}$ of size between $n^{1/4}/3$ and $3n^{1/4}/2$. 
Let $\zf=\zf_0$, and let $\{T_1,\ldots,T_R\}$, which we denote $\st$, be the set of indices $i$ with $T < i \leq n$ 
for which $|X_i|=\sqrt{n}$, ordered so that $T < T_1 < \ldots < T_R \leq n$. $R$ is distributed as $\Bin(n-T,n^{-3/4})$, 
so by Bayes' formula, (\ref{eq:chernoff_u}), and (\ref{eq:chernoff_l}),  
\bea
\p{\st \notin \zf,T\leq t^*} & =  & \sum_{t=1}^{t^*}\p{\st \notin \zf \vert T=t}\p{T=t}\eal 
                             & \leq & \sup_{t \leq t^*} \p{\st \notin \zf \vert T=t} = \sup_{t \leq t^*} \p{\st \notin \zf_t \vert T=t} \eal
                                                         &  =   & \sup_{t \leq t^*} \p{\Bin\left(n-t,\frac{1}{n^{3/4}}\right)<\frac{n^{1/4}}{3}\mbox{ or } 
                                                                                                                        \Bin\left(n-t,\frac{1}{n^{3/4}}\right)>\frac{3n^{1/4}}{2}} \eal 
                                                         &  =   & O\left(\frac{1}{n^6}\right).\label{eq:tsetbound}
\eea
This completes our bounds on the ``unexpected events''. We next use these inequalities to bound 
$\p{S_n=\sqrt{n},\eventn{Pos}_n}$. Roughly speaking, in order that $\{S_n=\sqrt{n}\}$ and $\eventn{Pos}_n$ 
occur, it is necessary that 
\begin{itemize}
\item[(a)] $S$ stays positive until time $T$,
\item[(b)] The random walk $S'$ given by $S'_i=\sum_{j=1}^i X_{T_j}/\sqrt{n}$ does not go ``too negative'' 
and additionally $|S'_R|$ is not ``too large'', and 
\item[(c)] $S_n-S_T-\sqrt{n}S'_R = \sqrt n-S_t-\sqrt{n}S_R'$. 
\end{itemize}
Though the event in (c) is precisely the event that $S_n=\sqrt{n}$, we write it in this form in order to point out that once we have conditioned on {\em fixed values} for $T$ and $\st$, $S_n-S_T-\sqrt{n}S'_R$ is independent of $\sqrt{n}-S_T-\sqrt{n}S'_R$. 
We now turn to the details of defining and bounding the events in (a)-(c). 

First, recall that $t^* = 5n^{3/4}\log n$. For any $t \leq t^*$, by Lemma \ref{lem:easy1}, 
\bea
\p{\eventn{Pos}_T,T=t}  &\leq & \p{\eventn{Pos}_{t-1},T=t} \eal
						& \leq & \p{\eventn{Pos}_{t-1},|X_t|=\sqrt{n}} \eal
                                                & =   & \p{\eventn{Pos}_{t-1}}\p{|X_t|=\sqrt{n}} \eal
                                                & =   & O\left(\frac{1}{\sqrt{t}}\right)\cdot O\left(\frac{1}{n^{3/4}}\right) = O\left(\frac{1}{\sqrt{t}n^{3/4}}\right)\label{eq:posttbound}
\eea
For {\em any} events $E,F$, and $G$, $\p{E} \leq \p{E,F}+\p{\bar{F}}$, and $\p{E,F} \leq \p{E,F,G}+\p{F,\bar{G}}$. 
We now apply these bounds together with the bounds (\ref{eq:bbound}), (\ref{eq:tgrowbound}), and (\ref{eq:stbbound_2}), to see that for any event $E$,  
\bea
\p{E,\eventn{Pos}_T} & \leq & \p{E,\eventn{Pos}_T,B} + o(2^{-n/2}) \eal
                                         & \leq & \p{E,\eventn{Pos}_T,B,T\leq t^*} + O(n^{-5}) \eal
                                         & \leq & \p{E,\eventn{Pos}_T,B,T\leq t^*,|S_{T-1}|\leq 8k\sqrt{t^*\log n}} \eal
                                         &              & \quad + \p{B,T\leq t^*,|S_{T-1}|> 8k\sqrt{t^*\log n}} + O(n^{-5}) \eal
                                         & =    & \p{E,\eventn{Pos}_T,B,T\leq t^*,|S_{T-1}|\leq 8k\sqrt{t^*\log n}} + O(n^{-5})\label{eq:epost_1}
\eea
Continuing in this fashion using (\ref{eq:tsetbound}), (\ref{eq:epost_1}), and the fact that $8k\sqrt{t^*\log n} \leq 20 n^{3/8}\log^{3/2}n$, 
and letting $j^*=20 n^{3/8}\log^{3/2}n$, we have 
\bea
\p{E,\eventn{Pos}_T} & \leq & \p{E,\eventn{Pos}_T,B,T\leq t^*,|S_{T-1}|\leq j^*,\st \in \zf} \eal
                                         &              & \quad + \p{\st \notin \zf,T \leq t^*} + O(n^{-5}) \eal
                                         & =    & \p{E,\eventn{Pos}_T,B,T\leq t^*,|S_{T-1}|\leq j^*,\st \in \zf} + O(n^{-5}).\eal
                                         & =    & \sum_{t=1}^{t^*}\p{E,\eventn{Pos}_T,B,T= t^*,|S_{T-1}|\leq j^*,\st \in \zf} + O(n^{-5}).\label{eq:epost_2}
\eea
By applying Bayes' formula and (\ref{eq:posttbound}), this yields 
\bea
\p{E,\eventn{Pos}_T} & \leq & \sum_{t=1}^{t^*}\p{\eventn{Pos}_t,T=t}\p{E\vert\eventn{Pos}_t,T=t,B,|S_{t-1}|\leq j^*,\st \in \zf} + O(n^{-5}) \eal
                                         &  =   & \sum_{t=1}^{t^*}O\left(\frac{1}{\sqrt{t}n^{3/4}}\right)\p{E\vert\eventn{Pos}_t,T=t,B,|S_{t-1}|\leq j^*,\st \in \zf_t} + O(n^{-5}) \label{eq:epost_3}
\eea
Next, for any $1 \leq t \leq t^*$ we have 
\[
\p{E\vert\eventn{Pos}_t,T=t,B,|S_{t-1}|\leq j^*,\st \in \zf_t} \leq \sup_{|s| \leq j^*,\ssi \in \zf_t} \p{E\vert\eventn{Pos}_t,T=t,B,S_{t-1}=s,\st =\ssi},\]
which together with (\ref{eq:epost_3}) gives 
\bea
\p{E,\eventn{Pos}_T} & \leq & \sum_{t=1}^{t^*} O\left(\frac{1}{\sqrt{t}n^{3/4}}\right) \sup_{|s| \leq j^*,\ssi \in \zf_t} \p{E\vert\eventn{Pos}_t,T=t,B,S_{t-1}=s,\st =\ssi} + O(n^{-5}) \eal
                                         & =    & O\left(\frac{\sqrt{t^*}}{n^{3/4}}\right)\sup_{1 \leq t \leq t^*}\sup_{|s| \leq j^*,\ssi \in \zf_t} \p{E\vert\eventn{Pos}_t,T=t,B,S_{t-1}=s,\st =\ssi} \eal
                                         &      & \qquad\qquad\qquad\qquad\qquad\qquad\qquad\qquad+ O(n^{-5}) \eal
                                         &  =   & O\left(\frac{\sqrt{\log n}}{n^{3/8}}\right)\sup_{1 \leq t \leq t^*}\sup_{|s| \leq j^*,\ssi \in \zf_t} \p{E\vert\eventn{Pos}_t,T=t,B,S_{t-1}=s,\st =\ssi}+O(n^{-5}) \eal\label{eq:epost_4}
\eea
We will apply equation (\ref{eq:epost_4}) with $E$ the event $\{S_n=n\}\cap \eventn{Pos}_n$. 
We first note that for a given $t$ with $1 \leq t \leq t^*$, if $|X_t|=\sqrt{n}$ and $|S_{t-1}|=s \leq j^* < \sqrt{n}$, then 
for $\eventn{Pos}_t$ to occur necessarily $X_t=\sqrt{n}$, so $S_t=\sqrt{n}+s$. 

Fix any integer $t$ with $1 \leq t \leq t^*$, any integer $s$ for which $|s|\leq j^*$, and any $\ssi \in \zf_t$. 
We hereafter denote by $\eventn{Good}$ the intersection of events 
\[\eventn{Pos}_t \cap \{T=t\} \cap B \cap \{S_{t-1}=s\} \cap \{\st =\ssi\},\]
and by $\psup{c}{\cdot}$ the conditional probability measure 
\[\p{\cdot \vert\eventn{Good}}.\] 
Given that $\{T_1,\ldots,T_R\}=\ssi$, $R$ is deterministic -- say $R=r$ -- and $n^{1/4}/3 \leq r \leq 3n^{1/4}/2$. We recall that $S'$ was the random walk with $S_i'=\sum_{j=1}^i X_{T_j}/\sqrt{n}$. 
As previously discussed, given that $\eventn{Good}$ occurs, $S_t=\sqrt{n}+s$. 
In order that $\{S_n=\sqrt{n}\}$ and $\eventn{Pos}_n$ occur, then, it is necessary that either 
\begin{enumerate}
        \item for some integer $m$ with $|m| \leq 10\log^2 n$, $S_r'=m$,  
        $S_j' \geq -10\log^2 n$ for all $1 \leq j \leq r$, and $S_n-S_t-\sqrt{n}S'_r=-s-m\sqrt{n}$ 
        (we denote these events $B_m$ for $|m| \leq 10\log^2 n$), or 
        \item there is $j$ with $t < j \leq n$ for which $|\sum_{i=1}^{k-1}S_{j,i}| \geq 10\sqrt{n}\log^2 n$ (we denote this 
        event $C$). 
\end{enumerate}
To see this, observe that if 
none of the events $B_m$ occurs and $C$ does not occur, then either: 
\begin{itemize}
\item $|S_r'|> 10 \log^2 n$, in which case 
\[S_n \geq S_r'\sqrt{n}-|\sum_{i=1}^{k-1}S_{n,i}| \geq S_r'-10\sqrt{n}\log^2 n> \sqrt{n},\]
so $S_n\neq \sqrt{n}$, or
\item there is $j$ with $1 \leq j \leq r$ for which $S_j' < -10\log^2 n$, in which case 
\[S_{T_j} \leq S_j'\sqrt{n}+|\sum_{i=1}^{k-1}S_{T_j,i}| \leq S_j'\sqrt{n}+10\sqrt{n}\log^2 n< 0,\]
so $\eventn{Pos}_n$ does not occur, or
\item $S_r'=m$ for some $m$ with $|m|\leq 10\log^2 n$, but $S_n-S_t-\sqrt{n}+s\neq -s-m\sqrt{n}$,
so $S_n\neq \sqrt{n}$. 
\end{itemize}
Thus, by a union bound, 
\begin{equation}\label{eq:pcond_spos}
        \psup{c}{S_n=\sqrt{n},\eventn{Pos}_n} \leq \psup{c}{\left(\bigcup_{|m| \leq 10\log^2 n} B_m\right)\cup C} \leq \psup{c}{C} + \sum_{m=-10\log^2 n}^{10\log^2 n} \psup{c}{B_m}.
\end{equation}
To bound the probabilities $\psup{c}{B_m}$, we first note that, denoting 
by $A_m$ the event that $S_r'=m$ and 
$S_j' \geq -10\log^2 n$ for all $1 \leq j \leq r$, 
\[B_m = A_m\cap\{S_n-S_t-\sqrt{n}S'_r=\sqrt{n}-s-m\sqrt{n}\}.\] 
Furthermore, $A_m$ and $\{S_n-S_t-\sqrt{n}S'_r=\sqrt{n}-s-m\sqrt{n}\}$ are 
independent as they are determined by disjoint sections of the random walk, so for all $m$ with $|m| \leq 10\log^2 n$, 
\bea
\psup{c}{B_m} & = & \psup{c}{A_m,S_n-S_t-\sqrt{n}S'_r=\sqrt{n}-s-m\sqrt{n}} \eal
                          & = & \psup{c}{A_m}\psup{c}{S_n-S_t-\sqrt{n}S'_r=\sqrt{n}-s-m\sqrt{n}}\label{eq:pcond_bj}
\eea
Now, given that $\eventn{Good}$ occurs, $S'$ is nothing but a symmetric simple random walk of length $r$; thus, by Bertrand's ballot theorem, 
\begin{equation}\label{eq:pcond_aj}
\psup{c}{A_m} = O\left(\frac{(m+10\log^2 n + 1)(10\log^2 n+1)}{r^{3/2}}\right) = O\left(\frac{\log^4 n}{n^{3/8}}\right).
\end{equation}
Also, given that $\eventn{Good}$ occurs, $S_n-S_t-\sqrt{n}S_r'$ is a sum of $n-t-r=\Omega(n)$ i.i.d.~integer-valued 
random variables that are never zero; Thus, by Theorem \ref{thm:spread}, 
\begin{equation}\label{eq:pcond_ntr}
\psup{c}{S_n-S_t-\sqrt{n}S'_r = \sqrt{n}-s-m\sqrt{n}} = O\left(\frac{1}{\sqrt{n}}\right), 
\end{equation}
and combining (\ref{eq:pcond_bj}), (\ref{eq:pcond_aj}), and (\ref{eq:pcond_ntr}) yields that 
\begin{equation}\label{eq:pcond_bj2}
        \psup{c}{B_m} = O\left(\frac{\log^4 n}{n^{7/8}}\right).
\end{equation}
For any $j$ with $t \leq j \leq n$, an argument just as that leading to 
(\ref{eq:ns_bound}) shows that 
\[\psup{c}{|\sum_{i=1}^{k-1}S_{j,i}| \geq 10\sqrt{n}\log^2 n} = O\left(\frac{1}{n^6}\right),\]
so
\begin{equation}\label{eq:pcond_c}
        \psup{c}{C} \leq \sum_{j=t}^n \psup{c}{|\sum_{i=1}^{k-1}S_{j,i}| \geq 10\sqrt{n}\log^2 n} = O\left(\frac{1}{n^5}\right),
\end{equation}
and (\ref{eq:pcond_spos}), (\ref{eq:pcond_bj2}), and (\ref{eq:pcond_c}) together yield 
\begin{equation}\label{eq:pcond_spos2}
        \psup{c}{S_n=\sqrt{n},\eventn{Pos}_n} = O\left(\frac{1}{n^5}\right)+\sum_{m=-10\log^2n}^{10\log^2n}O\left(\frac{\log^4 n}{n^{7/8}}\right) = O\left(\frac{\log^6 n}{n^{7/8}}\right).
\end{equation}
Since $t,s$, and $\ssi\in \zf_t$ were arbitrary, (\ref{eq:epost_4}) and (\ref{eq:pcond_spos2}) combine to give 
\[\p{S_n=\sqrt{n},\eventn{Pos}_n,\eventn{Pos}_T}=O\left(\frac{\sqrt{\log n}}{n^{3/8}}\right)\cdot O\left(\frac{\log^6 n}{n^{7/8}}\right)=O\left(\frac{\log^{13/2} n}{n^{5/4}}\right).\]
Finally, since if $\eventn{Pos}_n$ occurs then either $\eventn{Pos}_T$ occurs or $T>n$, by (\ref{eq:tgrowbound}) we have
\bea
\p{S_n=\sqrt{n},\eventn{Pos}_n} & \geq & \p{S_n=\sqrt{n},\eventn{Pos}_n,\eventn{Pos}_T} - \p{T > n} \eal
                                                                &  =   & O\left(\frac{\log^{13/2} n}{n^{5/4}}\right) - O\left(\frac{1}{n^5}\right) = O\left(\frac{\log^{13/2} n}{n^{5/4}}\right), \nonumber
\eea
as asserted in (\ref{eq:ce2_snpos}).

\section{Conclusion}\label{sec:conclusion}
The results of this paper raise several questions. 
While our examples show that Theorem \ref{thm:gbt_new} is essentially best possible, 
is it perhaps possible that a ballot theorem holds for real-valued Markov chains with finite variance? 
Also, one observation about the example of Section \ref{sec:other} 
is that in that example, the step size $X$ is not in the domain of attraction of 
{\em any} distribution. This leaves open the possibility that a ballot-style 
theorem may hold if the $X$ has mean zero and is in the domain of attraction 
of some stable law. Such behavior seems unlikely but is not ruled out by our examples. 

Finally, it would be very interesting to derive 
conditions on more general multisets $\ssf$ of $n$ numbers summing to some 
value $k$ which guarantee that, in a uniformly random permutation of $\ssf$, 
all partial sums are positive with probability of order $k/n$. 
Indeed, perhaps such work could end up not only generalizing the work of this paper, 
but perhaps unifying it with the existing discrete-time ballot theorems based 
on the ``increasing stochastic process'' perspective.

\bibliographystyle{plainnat}
\bibliography{bib_ballot}

\end{document}